\documentclass[11pt,a4paper,twoside]{amsart}

\usepackage{amsmath,amsfonts,amsthm,amsopn,color,amssymb,enumitem,cite,soul}
\usepackage{palatino}
\usepackage{graphicx}
\usepackage[normalem]{ulem}
\usepackage[colorlinks=true,urlcolor=blue,citecolor=red,linkcolor=blue,linktocpage,pdfpagelabels,bookmarksnumbered,bookmarksopen]{hyperref}
\hypersetup{urlcolor=blue, citecolor=red, linkcolor=blue}
\usepackage[left=3.4cm,right=3.4cm,top=2.5cm,bottom=2.5cm]{geometry}

\usepackage[hyperpageref]{backref}

\usepackage[colorinlistoftodos]{todonotes}
\makeatletter
\providecommand\@dotsep{5}
\def\listtodoname{List of Todos}
\def\listoftodos{\@starttoc{tdo}\listtodoname}
\makeatother
\allowdisplaybreaks

\numberwithin{equation}{section}

\newcommand{\e}{\varepsilon}
\newcommand{\eps}{\varepsilon}

\newcommand{\R}{\mathbb{R}}

\newcommand{\RN}{{\mathbb{R}^N}}
\newcommand{\RD}{{\mathbb{R}^2}}

\newcommand{\de}{\partial}
\newcommand{\weakto}{\rightharpoonup}

\DeclareMathOperator{\meas}{meas}

\renewcommand{\le}{\leslant}
\renewcommand{\ge}{\geslant}
\renewcommand{\a }{\alpha }

\renewcommand{\b }{\beta }

\renewcommand{\d }{\delta }

\newcommand{\g }{\gamma }

\renewcommand{\l }{\lambda}

\newcommand{\n }{\nabla }
\newcommand{\s }{\sigma }

\newcommand{\G}{\Gamma}

\newcommand{\X}{\mathcal{X}}

\newcommand{\calh}{\mathcal{H}}
\newcommand{\ch}{\mathcal{H}^{2,p}}
\newcommand{\chr}{\mathcal{H}_r^{2,p}}
\newcommand{\cx}{\mathcal{X}^{2,p}}
\newcommand{\cxr}{\mathcal{X}_r^{2,p}}
\renewcommand{\H}{H^1(\RD)}
\newcommand{\Hr}{H^1_r(\RD)}

\newcommand{\Ha}{\mathcal H}

\newcommand{\ird }{\int_{\RD}}

\def\bbm[#1]{\mbox{\boldmath $#1$}}
\newcommand{\beq }{\begin{equation}}
\newcommand{\eeq }{\end{equation}}

\renewcommand{\le}{\leqslant}
\renewcommand{\ge}{\geqslant}
\newcommand{\dis}{\displaystyle}

\def\br#1\er{\textcolor{red}{#1}} 
\def\bl#1\el{\textcolor{blue}{#1}}

\newtheorem{theorem}{Theorem}[section]
\newtheorem{lemma}[theorem]{Lemma}

\newtheorem{definition}[theorem]{Definition}
\newtheorem{proposition}[theorem]{Proposition}
\newtheorem{remark}[theorem]{Remark}

\newtheorem{counterexample}[theorem]{Counterexample}

\title[Planar Schr\"odinger equation with competing logarithmic self-interaction]{Schr\"odinger equation in dimension two \\ with  competing logarithmic self-interaction
}

\author[A. Azzollini]{Antonio Azzollini}
\address{A. Azzollini \newline\indent
Dipartimento di Matematica, Informatica ed Economia, \newline\indent Universit\`a degli
	Studi della Basilicata,
	\newline\indent
	Via dell'Ateneo Lucano 10, 85100
	Potenza, Italy}
\email{antonio.azzollini@unibas.it}

\author[P. d'Avenia]{Pietro d'Avenia}
\address{P. d'Avenia
\newline\indent Dipartimento di Meccanica, Matematica e Management,\newline\indent
	Politecnico di Bari
	\newline\indent
	Via Orabona 4,  70125  Bari, Italy}
\email{pietro.davenia@poliba.it}

\author[A. Pomponio]{Alessio Pomponio}
\address{A. Pomponio
\newline\indent Dipartimento di Meccanica, Matematica e Management,\newline \indent
	Politecnico di Bari
	\newline\indent
	Via Orabona 4,  70125  Bari, Italy}
\email{alessio.pomponio@poliba.it}

\subjclass[2010]{35J20, 35J60, 46E35}
\keywords{logarithmic convolution kernels; competing nonlocal terms; Schr\"odinger-Poisson or Choquard type system; weighted Sobolev spaces}

\begin{document}
	\maketitle
	\begin{abstract}
	In this paper we study the equation
			\[
	-\Delta u  +(\log |\cdot|*|u|^2)u=(\log|\cdot|*|u|^q)|u|^{q-2}u, \qquad \hbox{ in }\RD,
		\]
		where $8/3 < q < 4$. By means of variational arguments, we find infinitely many radially symmetric classical solutions. The main difficulties rely on the competition between the two nonlocal terms and on the presence of  logarithmic kernels, which have not a prescribed sign. In addition, in order to find finite energy solutions, a suitable functional setting analysis is required.
	\end{abstract}

	\section{Introduction}
	
	In the last years a wide literature has been devoted to the study of nonlinear partial differential equations involving nonlocal terms, such as Choquard, Schr\"odinger-Newton, Hartree-Fock, Schr\"odinger-Poisson equation. Such problems can be obtained  also considering systems arising from electromagnetic theory, gravitational theory, and quantum chemistry (see, for example, \cite{Lieb,MPT,Pekar}).
	In particular, we refer to nonlocal terms containing the convolution of a kernel (usually the Poisson kernel or, more in general, the Riesz kernel) with a nonlinear function. For instance, in the classical Choquard equation \cite{Lions}, the kernel is $ 1/{|x|}$ and the nonlinearity is $F(u)=|u|^2$ and so the equation reads as
	\begin{equation}\label{cho}
	-\Delta u  + u= \left(\frac 1{|x|}*|u|^2\right)u \qquad \hbox{ in }\RN, \text{ with }N\ge 1. 
	\end{equation}

	In literature there are many extensions of \eqref{cho}. For example, Moroz and Van Schaftingen in \cite{MVJFA} generalize the power of the nonlinearity $F$ and consider an arbitrary Riesz kernel $I_\a (x)=1/|x|^{N-\a}$, with $\a \in (0,N)$, studying 
	\begin{equation*}
	-\Delta u  + u= \left(\frac 1{|x|^{N-\a}}*|u|^p\right)|u|^{p-2}u \qquad \hbox{ in }\RN,\text{ with }N\ge 1.
	\end{equation*}

	Later on,  the same authors, in \cite{MVTrans} deal with a more general nonlinearity $F$, in the spirit of Berestycki-Lions \cite{BL1} and so the equation becomes  
	\begin{equation}\label{choa}
	-\Delta u  + u= \left(\frac 1{|x|^{N-\a}}*F(u)\right)f(u) \qquad \hbox{ in }\RN,\text{ with }N\ge 1
	\end{equation}
	where  $\a \in (0,N)$, $f:\R \to \R$ is a continuous function, and $\dis F(t)=\int_0^t f(s)ds$.

	Observe that, if $\alpha=0$ in \eqref{choa}, since $I_0\ast F(u)=F(u)$,  the growth conditions on $F$ in \cite{MVTrans} cover only partially the Berestycki-Lions assumptions, namely in the case of a nonnegative (attractive) nonlinearity.
	
	On the other hand, as for instance in the Hartree-Fock theory, the interaction potential could be also repulsive \cite{Benguria,LionsCMP}.
	Therefore \cite{dMP} analyses the following equation
	\begin{equation}\label{eqdmp}
	-\Delta u
	= \big(I_\alpha\ast F(u)\big)f(u)-\big(I_\beta\ast G(u)\big)g(u)
	\qquad \hbox{ in }\R^N,
	\end{equation}
	where $N\ge 3$, $0\le \b\le \a<N$, 
	$g:\R \to \R$ is a continuous function and 
	$\dis G(t)=\int_0^tg(s) ds.$
	In the limiting case $\alpha=\beta=0$, the growth conditions on $F$ and $G$ fully cover the Berestycki-Lions assumptions in \cite{BL1}.

	We finally refer to  \cite{MVSurvey} and the references therein for  similar problems.
	
	Since, as we said before, the convolution kernel can be the Poisson one, in dimension $2$, it make sense to replace in \eqref{cho} $1/|x|$ with $\log|x|$.  Thus
	Cingolani-Weth in \cite{CW} consider 
	\begin{equation}\label{sist-cw}
	\begin{cases}
	-\Delta u + u +\g\phi u=0  & \text{in }\RD,
	\\
	\Delta \phi=u^2 & \text{in }\RD,
	\end{cases}
	\end{equation}
	with $\g >0$. 
	Up to harmonic functions, the solution $\phi$ of the second equation is the convolution of the fundamental solution of the Laplacian in two dimensions with $u^2$. So \eqref{sist-cw} reads as the single nonlocal equation
	\begin{equation}\label{eqcw}
	-\Delta u +u +\g(\log |\cdot|*|u|^2)u=0, \qquad \hbox{ in }\RD.
	\end{equation}
	At least formally,
	solutions of  \eqref{eqcw} can be found as critical points of the functional
	\begin{equation}
	E(u)=\ird (|\n u|^2+u^2)\, dx + \frac \g{2}\ird \ird \log (|x-y|)|u(x)|^2|u(y)|^2\, dx \, dy.
	\end{equation}
	With respect to the Riesz kernel cases there are  important differences and several difficulties have to be faced.
	First of all, while the Riesz kernel is strictly positive, $\log|x|$ changes sign. For this reason, in \cite{CW,Stu}, the nonlocal term
	\begin{equation*}
	V(u)=\ird \ird \log (|x-y|)|u(x)|^2|u(y)|^2\, dx \, dy
	\end{equation*}
	is written as $V=V_1-V_2$ where
	\begin{align*}
	V_{1}(u)&=\ird \ird \log (1+|x-y|)|u(x)|^2|u(y)|^2\, dx \, dy,\\
	V_{2}(u)&=\ird \ird \log \left(1+\frac 1{|x-y|}\right)|u(x)|^2|u(y)|^2\, dx \, dy.
	\end{align*}
	Observe $V_2$ is finite in $\H$, by Hardy-Littlewood-Sobolev inequality, but, to guarantee the finiteness of $V_1$, as in \cite{Stu}, one has to consider the smaller Hilbert space
	\[
	X=\left\{u\in \H: \ird \log(1+|x|)|u(x)|^2 \, dx<+\infty\right\}.
	\]
	So the functional $E$ is well defined in $X$.
	\\
	In \cite{CW,DW}, the authors assume  $\g >0$. In this case, the functional has good geometric properties, namely the Mountain Pass structure. Instead, whenever $\g <0$, the situation is completely different and more complicated. Indeed the term $V_1$ appears with a negative coefficient in the functional $E$ which turns out to be strongly indefinite, namely unbounded from above and below on an infinite-dimensional subspace of $X$ and this indefiniteness cannot be removed by a compact perturbation. 
	In order to control the term $V_1$, since it contains a logarithmic type potential, Azzollini in \cite{A} (see also \cite{APim}) introduces a potential with a polynomial growth near the linear term, studying
	\[
	-\Delta u +(1+|x|^{\a})u -(\log |\cdot|*|u|^2)u=0, \qquad \hbox{ in }\RD,
	\]
	with $\a >0$.
	
	In this paper our aim is to deal with a logarithmic kernel in dimension $2$ in presence of two nonlocal terms in competition  as in \eqref{eqdmp}, namely
	\begin{equation}\label{eq}\tag{$\mathcal{P}$}
	-\Delta u  +(\log |\cdot|*|u|^2)u=(\log|\cdot|*|u|^q)|u|^{q-2}u, \qquad \hbox{ in }\RD,
	\end{equation}
	where $q>2$.

	
	Formally, solutions of \eqref{eq} can be found as critical points of the functional
	\begin{multline*}
	I(u): =\ird |\n u|^2\, dx + \frac 1{2}\ird \ird \log (|x-y|)|u(x)|^2|u(y)|^2\, dx \, dy\\
	-\frac 1{q}\ird \ird \log (|x-y|)|u(x)|^q|u(y)|^q\, dx \, dy,
	\end{multline*}
	and it is immediately clear that it carries somehow all the difficulties run into \eqref{eqcw} both in the case $\g >0$ and $\g <0$. To be more precise, introducing the following notations, for $k\ge2$,
	\begin{equation*}
	V_{0,k}(u):=\ird \ird \log (|x-y|)|u(x)|^k|u(y)|^k\, dx \, dy
	\end{equation*}
	and 
	\begin{align*}
	V_{1,k}(u)&:=\ird \ird \log (2+|x-y|)|u(x)|^k|u(y)|^k\, dx \, dy,\\
	V_{2,k}(u)&:=\ird \ird \log \left(1+\frac 2{|x-y|}\right)|u(x)|^k|u(y)|^k\, dx \, dy,
	\end{align*}
	then
	\begin{equation}
	\label{Vkrel}
	V_{0,k}=V_{1,k}-V_{2,k}
	\end{equation}
	and we can write
	\begin{align*}
	I(u)&= \ird |\n u|^2\, dx + \frac 1{2} V_{0,2}(u)
	-\frac 1{q}V_{0,q}(u)
	\\
	&=\ird |\n u|^2\, dx
		+\frac{1}{2}V_{1,2}(u)
		-\frac{1}{2}V_{2,2}(u)
		-\frac{1}{q}V_{1,q}(u)
		+\frac{1}{q}V_{2,q}(u).
	\end{align*}
	Also in this case, by
	Hardy-Littlewood-Sobolev inequality, $V_{2,k}$ is finite in $\H$, for $k\in\{2,q\}$. At contrary, as already explained, $V_{1,k}$ is not well defined in $\H$, for $k\in\{2,q\}$. To control $V_{1,2}$, similarly as in \cite{Stu}, we introduce the space $\X$ as the completion of $C_0^{\infty}(\RD)$ with respect to the norm $$\|\cdot\|_{\X}:=
	\sqrt{\|\n \cdot\|_2^2+\|\cdot\|_{*,2}^2}\ ,$$
	where, for any $k\ge 2$, 
	\begin{equation}\label{norma*k}
	\|u\|_{*,k}:=\left(\ird \log(2+|x|)|u(x)|^k \, dx\right)^{\frac 1k}. 
	\end{equation}

We say that a function $u$ has {\em finite energy} if it belongs to $\X$.
	
	At this point it is not clear at all if also $V_{1,q}$ is well defined in $\X$.

Observe that $\X$ coincides with the weighted Sobolev space $\Ha_V$ where the weight is $V(x)=\log(2+|x|)$, while 
	$V_{2,q}$ is well defined in weighted Lebesgue space $L_V^q(\RD)$ (see Section \ref{funct} for precise definitions).
 A wide literature has been devoted to the study of weighted Sobolev and Lebesgue spaces and their embedding properties (see, for example, \cite{OK} and the references therein). However, up to our knowledge, there is no result concerning the embedding properties of $\Ha_V$ into $L_V^q(\RD)$, namely of a weighted Sobolev space into weighted Lebesgue spaces with the {\em same coercive weight}.  We  will need to examine in depth this new functional setting to discover nice properties guaranteeing  a good definition of our functional $I$. In particular, we will provide a  Sobolev type continuous embedding theorem, but also a counterexample to show that such an embedding is not compact.
	
	However, the analysis of the geometry of the functional on one hand and its compactness properties on the other will shift the focus to the problem of finding a suitable estimate permitting to control the term $V_{1,q}$ not just by the $\mathcal X-$norm, but also by the $H^1-$norm. We will prove the existence of such an estimate making use of the well known Strauss Lemma in the set of radial functions of $\X$.
	
Our main result is
	
	\begin{theorem}\label{main}
		If  $8/3<q<4$, then there exist infinitely many  nontrivial radial classic solutions with finite energy to problem \eqref{eq}, whose corresponding critical values diverge.
	\end{theorem}

We remark that the assumption  $q\in (8/3,4)$ is purely technical and it is necessary to obtain the boundedness of  Cerami sequences (see Proposition \ref{pr:C} and Lemma \ref{lebound}). It would be interesting to consider the other cases.
	
The paper is divided as follows. In Section \ref{funct}, we introduce the functional setting showing the related embeddings properties. Section \ref{3} is devoted to the proof of our main theorem. Finally, in Section \ref{4}, we generalize our previous arguments considering 
		\begin{equation}\label{eqp}\tag{$\mathcal{P}_p$}
		-\Delta u  +(\log |\cdot|*|u|^p)|u|^{p-2}u=(\log|\cdot|*|u|^q)|u|^{q-2}u, \qquad \hbox{ in }\RD,
		\end{equation}
		and extending our existence result to the case $1<p<q$ and $2p^2/(p+1)<q<2p$. 
		
		In the following we will use the symbol $\|\cdot\|$ to denote the norm of $H^1(\RD)$. By $B_R(x,y)$ and $B_R$ we indicate the ball of radius $R$ centred in $(x,y)\in \RD$ and in the origin, respectively.

	\section{Functional framework and preliminary results}\label{funct}
	
	The following two results are related with weighted Lebesgue and Sobolev  spaces. In particular the first is an interpolation lemma and the second is a continuous embedding theorem. 
	
	Even if there is a rich literature on weighted Lebesgue and Sobolev spaces, up to our knowledge it is the first time that coercive weights are considered in such a type of theorems.
	
	In the following,  $V:\RD\to \R$ is a continuous and  positive  and for $1\le \tau<+\infty$ we denote by $L^\tau_V(\RD)$ the weighted Lebesgue space endowed with the norm
	\[
	\|u\|_{V,\tau}:=\left(\ird V(x)|u|^\tau  dx \right)^{\frac 1\tau}.
	\]

	\begin{lemma}\label{le:inter}
		Let $V:\RD\to \R$ is a continuous and  positive.
		The following continuous embedding holds for $2\le q<\tau< r<+\infty$ 
		$$L^q_V(\RD)\cap L^r_V(\RD)\hookrightarrow L^\tau_V(\RD).$$ 
	\end{lemma}
	
	\begin{proof}
		Let $\tau \in (q,r)$ and consider $\l \in (0,1)$ such that $\tau=\l r+(1-\l)q$. Applying the Young inequality, We have
		\begin{align*}
		\|u\|^\tau_{V,\tau}
		&
		=\ird (V(x)|u|^r)^{\l }(V(x)|u|^q)^{1-\l}\, dx
		\le \|u\|_{V,r}^{\l r}\|u\|_{V,q}^{(1-\l) q}
		\\
		&\le\frac{\l r}\tau \|u\|^\tau_{V,r}+\frac{(1-\l) q}\tau\|u\|^\tau_{V,q},
		\end{align*}
		from which we conclude by the inequality $\sqrt[\tau]{a+b}\le \sqrt[\tau]{a}+\sqrt[\tau]{b}$ which holds for every $a\ge 0$ and $b\ge 0$ and $\tau\ge 2$.
	\end{proof}
	
	Now we define $\mathcal H_V$ as the completion of $C_0^{\infty}(\RD)$ with respect to the norm 
	$$\|u\|_V:=\sqrt{\|\n u\|_2^2+\|u\|^2_{V,2}}.$$
	
	\begin{theorem}\label{th:emb}
		Let $V:\RD\to \R$ is a continuous and  positive.		If the distributional derivatives of  $V$ are functions satisfying for a suitable $C>0$,
		\begin{equation}\label{gradv}
		|\n V(x)|\le C V^{\frac 32}(x),\quad \hbox{ a.e. in }\RD,
		\end{equation}   
		then for all $\tau\ge 2$ the following continuous embedding holds 
		$$\mathcal H_V \hookrightarrow L^\tau_V(\RD).$$ 
	\end{theorem}
	\begin{proof}
		We will follow the scheme in \cite[Theorem 9.9]{B}. Consider a compactly supported function $w$ absolutely continuous with respect to both the variables.
		\\
		We have that, for any  $x =(x_1,x_2)\in\RD$,
		\begin{align*}
		|w(x)|&=\left|\int_{-\infty}^{x_1} \frac{\partial w}{\partial x_1}(t,x_2)\, dt\right|\le  \int_{-\infty}^{+\infty} \left|\frac{\partial w}{\partial x_1}(t,x_2)\right|\, dt\\
		|w(x)|&=\left|\int_{-\infty}^{x_2} \frac{\partial w}{\partial x_2}(x_1,t)\, dt\right|\le \int_{-\infty}^{+\infty} \left|\frac{\partial w}{\partial x_2}(x_2,t)\right|\, dt,
		\end{align*}
		and then, as in \cite{B},
		\begin{equation}\label{eq:Brezis}
		\|w\|_{2}\le \left\|\frac{\partial w}{\partial x_1}\right\|^{\frac 12}_{1}\left\|\frac{\partial w}{\partial x_2}\right\|^{\frac 12}_{1}.
		\end{equation}
		Now we consider $u\in C^1(\RD)$ compactly supported and  observe that, by \eqref{gradv}, $\sqrt{V(x)}u^2$ is absolutely continuous with respect to both the variables.
		Since by \eqref{gradv}, for $i=1,2$ 
		\begin{align*}
		\left\|\frac{\partial (\sqrt{V(x)}u^2)}{\partial x_i}\right\|^{\frac 12}_{1}
		&\le \left\|\frac{ V_{x_i}(x)}{2\sqrt{V(x)}}u^2\right\|_{1}^{\frac12}+\left\|2\sqrt{V(x)}uu_{x_i}\right\|_{1}^{\frac 12}\\
		&\le  C(\|u\|_{V,2}+\|u\|^{\frac 12}_{V,2}\|\n u\|_{2}^{\frac 12})
		\le C\|u\|_{V},
		\end{align*}
		using \eqref{eq:Brezis}  for $w=\sqrt{V(x)}u^2$, we have
		\[
		\|u\|_{V,4}^2\le C\|u\|^2_{V}.
		\]
		By density, we deduce that $\calh_V$ is continuously embedded into $L^4_V(\RD)$ and so, by Lemma \ref{le:inter}, into  $L^\tau_V(\RD)$, for $\tau\in [2,4]$.
		\\
		By analogous computations applied to $w=\sqrt{ V(x)} |u|^m$ for some natural number $m\ge 2$, we have that
		\begin{align*}
		\|u\|_{V,2m}^m
		&\le  C(\|u\|_{V,m}^m+\|u\|^{m-1}_{{V,2m-2}}\|\n u\|_{2})
		\\
		&\le  C(\|u\|_{V,m}^m+\|u\|^{m}_{{V,2m-2}}+\|\n u\|_{2}^m)
		\end{align*}
		and so
		\[
		\|u\|_{V,2m}
		\le   C(\|u\|_{V,m}+\|u\|_{{V,2m-2}}+\|\n u\|_{2}).
		\]
		So, if $u\in L^{2m-2}_V(\RD)$ then, by  Lemma \ref{le:inter}, we deduce that $u\in L^{2m}_V(\RD)$, too. By induction,  a density argument allows us to  conclude.
	\end{proof}

	\begin{remark}
		Some words about assumption \eqref{gradv} are in order. This hypothesis essentially prevents the potential $V$ being {\em too oscillating} at infinity. For example $V(x)=|x|^2(\sin (e^{|x|})+2)+1$ does not satisfies \eqref{gradv}.
	\end{remark}

	\begin{remark}
		As a natural question in this preliminary analysis of the functional setting, we are interested in understanding if the continuous embedding we have just proved actually turns out to be also compact. We remark that the presence of a coercive potential in the norm of the Sobolev space $\mathcal H_V$ guarantees the compact embedding of $\mathcal H_V$ in all the Lebesgue spaces $L^\tau(\RD)$, for $\tau\ge 2$, exploiting the classical result proved by Rabinowitz \cite{R}.
		\\
		The presence of the same coercive potential as a weight also in  the Lebesgue space changes significantly the matter as recently shown in \cite{APS} through a counterexample. Indeed, therein, the following counterexample has been proved. We write the details for completeness.

	\end{remark}

	\begin{counterexample}\label{contro}
		Let $V:\RD\to \R$ be a positive coercive uniformly continuous function. Then $\calh_V$ is not compactly embedded into $L^\tau_V(\RD)$, for $\tau>2$.
	\end{counterexample}
	
	\begin{proof}
		In the following, for $n\ge 1$, we denote $\mathcal{B}'_n:=B_{\frac 1{2\sqrt{V(n,0)}}}(n,0)$, $\mathcal{B}''_n:=B_{\frac 1{\sqrt{V(n,0)}}}(n,0)$  and $\mathcal{A}_n:=\mathcal{B}_n''\setminus \mathcal{B}_{n}'$.
		For any $n\ge1$, we consider $u_n\in C(\RD)$ such that
		\[
		u_n(x):=
		\begin{cases}
		1 &\text{for }x\in \mathcal{B}_{n}',
		\\
		0 &\text{for }x\in \RD\setminus \mathcal{B}_{n}'',
		\end{cases}
		\]
		and with $|\n u_n(x)|=2\sqrt{V(n,0)}$, for  $x\in\mathcal{A}_{n}$.
		\\
		We have that
		\begin{align*}
		\ird |\n u_n|^2\, dx
		&=\int_{\mathcal{A}_n}|\n u_n|^2\,dx
		=4V(n,0) \cdot \meas(\mathcal{A}_n)\simeq c>0,
		\\
		\ird V(x)u_n^2\,dx
		&\le \int_{\mathcal{B}_n''}V(x)\,dx
		\le\max_{\overline{\mathcal{B}_n''}}V\cdot \meas(\mathcal{B}_n'')\simeq c>0,
		\\
		\ird V(x)|u_n|^\tau\,dx
		&\ge \int_{\mathcal{B}_n'}V(x)\,dx
		\ge\min_{\overline{\mathcal{B}_n'}}V \cdot \meas(\mathcal{B}_n')\simeq c'>0.
		\end{align*}
		Therefore, $\{u_n\}$ is a sequence of $\calh_V$ such that $0<c_1\le \|u_n\|_V\le c_2$, $u_n\weakto 0$ in $\calh_V$ but $\|u_n\|_{V,\tau}\ge c_3>0$. This shows that $\calh_V$ is not compactly embedded into $L^\tau_V(\RD)$.
	\end{proof}

	In our specific situation, we easily observe that the weight $V(x)=\log (2+|x|)$  fulfils assumptions of Theorem \ref{th:emb} and Counterexample \ref{contro}, and
	$$\X=\calh_{\log(2+|\cdot|)}\quad\|\cdot\|_{*,2}=\|\cdot\|_{\log(2+|\cdot|),2}.$$ 
	Thus, taking into account the well know result in \cite{R}, it is immediately proved  the following
	\begin{proposition}\label{embX}
		For all $\tau\in [2,+\infty)$,  $\X$ is continuously (but not compactly) embedded into $L^\tau_{\log(2+|\cdot~\!|)}(\RD)$  and compactly embedded into $L^\tau(\RD)$.
	\end{proposition}


	The following proposition shows that the functional $I$ is well defined  in $\X$.
	\begin{proposition}\label{levikx}
		For every $u\in \X$ and $k\ge 2$, the following inequalities hold
		\begin{align}
		0\le V_{1,k}(u)&\le C \|u\|_{*,k}^k\|u\|_k^k,\label{control 1}
		\\
		0\le V_{2,k}(u)&\le C \|u\|_{\frac{4k}3}^{2k}.\label{control 2}
		\end{align}
		In particular, there exists a positive constant $C>0$ such that,
		for every $u\in \X$, we have
		\begin{equation*}\label{eq:cont1}
		V_{i,k}(u)\le C\|u\|^{2k}_{\X}, \quad \hbox{ for }k\ge 2 \hbox{ and } i\in\{1,2\}.
		\end{equation*}		
	\end{proposition}
	\begin{proof}
		Let $u\in \X$. Then 
		\begin{align*}
		V_{1,k}(u)&\le \ird\ird \big(\log(2+|x|)+\log(2+|y|)\big)|u(x)|^k|u(y)|^k\,dx\,dy
		=2\|u\|_{*,k}^k\|u\|_k^k,
		\\
		V_{2,k}(u)&\le 2\ird \ird \frac 1{|x-y|}|u(x)|^k|u(y)|^k\,dx\,dy\le C \|u\|_{\frac{4k}3}^{2k}.
		\end{align*}
		The last statement follows by Proposition \ref{embX}.
	\end{proof}
	
	Arguing as in \cite[Lemma 2.2]{CW} 
	we can prove the following
	\begin{proposition}
		The functionals  $V_{1,k}$ and $V_{2,k}$, with $k\ge 2$, and so also $I$, are of class $C^1$ in $\X$.
		Moreover critical points of $I$ are weak solutions of \eqref{eq}.
	\end{proposition}

		\begin{lemma}\label{le:wk}
			For $k\in \{2,q\}$, the function $w_k:\RD \to \R$, defined as 
			$$w_k(x):=\ird \log(|x-y|)|u(y)|^kdy$$
			satisfies $\Delta w_k=2\pi |u|^k$ and 
			\begin{equation}\label{wk}
			w_k(x)-\|u\|_k^k\log|x|\to 0, \quad |x|\to +\infty.
			\end{equation}
		\end{lemma}
		
		\begin{proof}
			As already observed in \cite[Lemma 2.3]{CW}, $w_2$ is locally bounded and satisfies \eqref{wk}. Let us prove that the same holds also for $w_q$. Only some of the arguments are those of \cite[Lemma 2.3]{CW} are the same, and so we write all the details for completeness.
			\\
			If $|x|\le 1$, we have
			\begin{equation*}
			|w_q(x)|\le \int_{B_2(x)} \big|\log(|x-y|)\big||u(y)|^qdy
			+\int_{\RD\setminus B_2(x)} \big|\log(|x-y|)\big||u(y)|^qdy
			\end{equation*}
			Since $1 \le|x-y| \le 2+|y|$ for $y \in \mathbb{R}^2  \setminus B_2(x)$, we find that
			$$
			\int_{\RD\setminus B_2(x)} \log (|x-y|)|u(y)|^q d y \le\|u\|_{*,q}^q .
			$$
			While
			\begin{equation*}
			\int_{B_2(x)} \big|\log(|x-y|)\big||u(y)|^qdy
			\le \left( \int_{B_2} \big|\log(|y|)\big|^2dy\right)^{\frac 12}\|u\|_{2q}^q<+\infty.
			\end{equation*}
			So $w_q\in L^\infty(B_1)$. 
			\\
			Let's now consider $|x|\ge 1$. Observe that
			$$
			w_q(x)-\|u\|_q^q\log |x|=\ird h(x, y) u^2(y) d y$$ with  
			$$h(x, y)=\log |x-y|-\log |x|=\log \frac{|x-y|}{|x|} .
			$$
			Note that $h(x, y) \rightarrow 0$, as $|x| \rightarrow \infty$ for every $y \in \mathbb{R}^2$. Moreover,
			$$
			\log \frac{1}{2} \le h(x, y) \chi_{\left\{|y-x| \ge \frac{1}{2}|x|\right\}}(y)\le \log (1+|y|) \quad\text { for all } x, y \in \mathbb{R}^2 \text { with }|x| \ge 1.
			$$
			Since $\log \frac{1}{2}|u|^q$ and $\log (1+|y|) |u|^q$ are in $L^1(\RD)$, we deduce that 
			\begin{equation}\label{h1}
			\int_{\left\{|y-x| \ge \frac{1}{2}|x|\right\}} h(x, y) |u(y)|^qdy \to 0, \qquad \text{as }|x|\to +\infty.
			\end{equation}
			While, since for $|y-x| \le \frac{1}{2}|x|$, we have
			\[
			|x|\le 2|y|\le 2(1+|y|),
			\]
			and so
			\begin{align*}
			0&\le \log|x|\int_{\left\{|y-x| \le \frac{1}{2}|x|\right\}} |u(y)|^qdy 
			\le \log|x|\int_{\left\{|y| \ge \frac{1}{2}|x|\right\}} |u(y)|^qdy 
			\\
			&\le \int_{\left\{|y| \ge \frac{1}{2}|x|\right\}}\log\big(2(1+|y|)\big) |u(y)|^qdy \to 0, \qquad \text{as }|x|\to +\infty.
			\end{align*}
			Moreover
			\begin{align*}
			0&\le \int_{\left\{|y-x| \le \frac{1}{2}|x|\right\}} \big|\log|x-y|\big||u(y)|^qdy 
			\\
			&\le \int_{\left\{|y| \ge \frac{1}{2}|x|\right\}}\log(1+|y|) |u(y)|^qdy 
			\to 0, \qquad \text{as }|x|\to +\infty.
			\end{align*}
			These last inequalities imply that
			\begin{equation}\label{h2}
			\int_{\left\{|y-x| \le \frac{1}{2}|x|\right\}} h(x, y) |u(y)|^qdy \to 0, \qquad \text{as }|x|\to +\infty.
			\end{equation}
			Now \eqref{wk} follows by \eqref{h1} and \eqref{h2}.
		\end{proof}
		
		\begin{proposition}\label{regola}
			If $u\in \X$ is a weak solution to \eqref{eq}, then $u$ is a classical solution.
		\end{proposition}
		
		\begin{proof}
			Denoting by $f:=w_q|u|^{q-2}u -w_2u$, since, by Lemma \ref{le:wk}, $w_k\in L^\infty_{\rm loc}(\RD)$ for $k\in \{2,q\}$, we deduce that $f\in L^\tau_{\rm loc}(\RD)$, for all $\tau\in [1,+\infty)$ and so, by \cite[Theorem 8.8]{GT}, $u\in W^{2,2}_{\rm loc}(\RD)$. Moreover, by \cite[Theorem 10.2]{LL}, $w_k\in C^{1,\a}_{\rm loc}(\RD)$, for $k\in \{2,q\}$ and $\a \in (0,1)$. This implies that $f\in W^{1,2}_{\rm loc}(\RD)$ and so, using \cite[Theorem 8.10]{GT}, we conclude that $u\in W^{3,2}_{\rm loc}(\RD)$ and so $u\in C^{2,\a}_{\rm loc}(\RD)$.
		\end{proof}

	In order to find a solution at some minimax level, in the next sections we are going to explore compactness  and geometrical properties of the functional $I$ in $\X$. 
	
	We point out the remarkable fact that a nontrivial obstacle to a simple direct application of linking theorems is the lack of a lower estimate on  the principal part of the functional 
	\begin{equation*}\label{parteprin}
	\mathcal{A}(u):=\|\n u\|_2^2 +\frac12 V_{1,2}(u)
	\end{equation*} 
	by the $\X $-norm. 
	
	Indeed observe that, differently from what happens in \cite{CW,DW}, in order to study the behaviour of the functional we will need to compare the principal part of the functional not with the integral of a pure power, but with the competing term $V_{1,q}$, for which only an upper estimate  by $\X$-norm is provided (see formula \eqref{control 1}). 
	
	In order to overcome this difficulty and obtain an upper estimate of $V_{1,q}(u)$ by $\mathcal A(u)$, we observe that obviously $\X\hookrightarrow \H$, and exploit  the invariance of the weight with respect to the action of  the orthogonal group for restricting our investigation to the set $\X_r$ of radial functions in $\X$,  where the following well known Strauss Lemma \cite{S} holds
	\begin{lemma}\label{strauss}
		There exists $C> 0$  such that for all $u \in \Hr$,  there holds
		\[
		|u(x)| \le C|x|^{-\frac{1}2} \|u\|, 
		\]
		for all $|x|\ge 1$.
	\end{lemma}

	As an immediate consequence we have the following

	\begin{lemma}\label{embHr}
		For all $\tau\in (2,+\infty)$,  there exists a constant $C$ such that for any $u\in \Hr$,
		\begin{equation*}
		\|u\|_{*,\tau}^\tau
		\le C\left(\|u\|_\tau^\tau+\|u\|^{\tau-2}\|u\|_2^2\right)\le C \|u\|^\tau.
		\end{equation*}
		As a consequence the space $\Hr$ is continuously embedded into $L^\tau_{\log(2+|\cdot|)}(\RD)$, for any $\tau\in (2,+\infty)$.

	\end{lemma}
	
	\begin{proof}
		Let $u\in \Hr$. Then, by Lemma \ref{strauss}, 
		\begin{align*}
		\ird\log(2+|x|)|u|^\tau\, dx
		&\le  C\left(\int_{B_1} |u|^\tau\, dx+
		\|u\|^{\tau-2}\int_{\RD\setminus B_1} \frac{\log(2+|x|)}{|x|^{\frac{\tau-2}2}}|u|^2\, dx\right)\nonumber\\
		&\le C\left(\|u\|_\tau^\tau+\|u\|^{\tau-2}\|u\|_2^2\right)
		\le C \|u\|^\tau.
		\end{align*}
	\end{proof}

	\begin{proposition}\label{cov1}
		For every $u\in \Hr$ and $k> 2$, the following inequalities hold
		\begin{equation*}
		0\le V_{1,k}(u) \le   
		C\|u\|_k^k\big(\|u\|_k^k+\|u\|^{k-2}\|u\|_2^2\big)\le C\|u\|^{2k}.
		\end{equation*}
		As a consequence for every $u\in \Hr$ and $k> 2$ we have
		\begin{equation*}\label{eq:cont2}
		0\le V_{1,k}(u)\le C(\mathcal A(u)+1)^k.
		\end{equation*}
	\end{proposition}

	\begin{proof}
		Set $u\in \Hr$. Then 
		\begin{align*}
		V_{1,k}(u) &
		\le  \ird\ird\big(\log(2+|x|)+\log(2+|y|)\big)|u(x)|^k|u(y)|^k\, dx\,dy=2 \|u\|_k^k\|u\|_{*,k}^k
		\end{align*}
		and we conclude by Lemma \ref{embHr}.\\
		The second part of the proposition follows from 
		the fact that
		\begin{equation}\label{eq:controlL2}
		V_{1,2}(u)\ge \log 2 \|u\|_2^4 \qquad\hbox{ for any $u\in\X$},
		\end{equation}
		and then $\|u\|_2^2\le C \sqrt{V_{1,2}(u)}\le C(V_{1,2}(u)+1).$
	\end{proof}

	\begin{remark}
		
		We point out also that, by the Palais' Principle of Symmetric Criticality, in order to solve \eqref{eq} we will be allowed to look for critical points of the constrained functional $I_{|\X_r}$.
	\end{remark}

	The following technical lemmas will be useful in the sequel. 
	\begin{lemma}[Lemma 2.1 of \cite{CW}]\label{lemma2.1}
		Let $\{u_n\}$ be a sequence in $L^2(\mathbb{R}^2)$ such that $u_n \rightarrow u \in L^2(\mathbb{R}^2) \backslash\{0\}$ pointwise a.e. on $\mathbb{R}^2$. Moreover, let $\{v_n\}$ be a bounded sequence in $L^2(\mathbb{R}^2)$ such that
		$$
		\sup _{n \in \mathbb{N}} \ird \ird \log(2+|x-y|)u_n^2(x) v_n^2(y)\, dx\,dy<+\infty.
		$$
		Then there exists $n_0 \in \mathbb{N}$ and $C>0$ such that $\|v_n\|_{*,2}<C$, for $n \ge n_0$.
		\\
		If, moreover,
		$$
		\ird \ird \log(2+|x-y|)u_n^2(x) v_n^2(y)\, dx\,dy\to 0 \quad\text { and }\quad \|v_n\|_2 \to 0 ,\quad\text { as } n \to +\infty,
		$$
		then
		$$
		\|v_n\|_{*,2} \to 0 \quad\text { as } n \to +\infty, \ n \ge n_0 .
		$$
	\end{lemma}
	
	\begin{lemma}[Lemma 2.6 of \cite{CW}]\label{lemma2.6}
		Let $\{u_n\}, \{v_n\}, \{w_n\}$ be bounded sequences in $\X$ such that $u_n\weakto u$ weakly in $\X$. Then, for every $z\in \X$, we have 
		\[
		\ird \ird \log(2+|x-y|)v_n(x)w_n(x)z(y) \big(u_n(y)-u(y)\big)\, dx\,dy\to 0 ,\quad\text { as } n \to +\infty.
		\]
	\end{lemma}
	
	%
	%
	%
	
	\section{Proof of Theorem \ref{main}}\label{3}

In this section we assume that $q\in (8/3,4)$.
	
 In order to find critical minimax levels, we develop arguments based on the augmented dimension trick as described in \cite{HIT} and used for a related problem in \cite{A}.
		
		As a first step, we introduce the comparison $C^1$ functional $G:\X_r\to\R$, 
		\begin{equation*}
		G(u):=\|\nabla u\|_2^2
		+\frac{1}{2}V_{0,2}(u)
		-\frac{1}{q}V_{1,q}(u)
		\text{ for } u\in\X_r.
		\end{equation*}
		Observe that for all $u\in \X_r$  we have $G(u)\le I(u)$.
		
The functionals $I$ and $G$ satisfy the following geometrical properties.
		
		\begin{proposition}\label{geo}
			There exist $\rho>0$, $\mu > 0$ and, for any $n\ge 1$, an odd continuous mapping 
			\begin{equation*}
			\g_{0n}:S^{n-1}\to \X_r
			\end{equation*}
			where $S^{n-1}=\{\sigma\in\R^n\mid |\sigma|=1\}$ such that
			\begin{enumerate}
				\item \label{g1} ${I}(u)\ge G(u)\ge 0$ for all $u\in \X_r$ with $\|u\|\le\rho$ and, in particular,
				${I}(u)\ge G(u)\ge \mu$ for all $u\in \X_r$ with $\|u\|=\rho$;
				\item \label{g2}$G(\g_{0n}(\sigma))\le I(\g_{0n}(\sigma))<0,$  for all $\sigma\in S^{n-1}$.
			\end{enumerate}
		\end{proposition}
		
		\begin{proof}
			Since, by \eqref{control 2}, Gagliardo-Nirenberg inequality, and Young inequality,
			\[
			V_{2,2}(u)
			\le C \|u\|_{\frac{8}{3}}^{4}
			\le  C\|\n u\|_2\|u\|_2^3
			\le \e \|\n u\|_2^{2}+ C_\e\|u\|_2^{6},
			\]
			using Proposition \ref{cov1} and \eqref{eq:controlL2},  we have 
	
			\begin{equation*}
			G(u)\ge c_1\|\n u\|_2^2+c_2\|u\|_2^{4}
			- \frac{C_\e}{2}\|u\|_2^{6}
			-C\|u\|^{2q}.
			\end{equation*}
			Therefore, if $\rho$ is sufficiently small,  
			\[
			G(u)\ge c_1\|u\|^{4}
			- c_2\|u\|^{6}
			-c_3\|u\|^{2q}
			\]
			and we get (\ref{g1}).\\
			As to the second geometrical property, we consider $n\ge 1$ and  $u_1,\ldots, u_n\in \X_r$ linearly independent and for any $\sigma=(\sigma_1,\ldots,\sigma_n) \in S^{n-1}$ we set $$u_\sigma:=\sum_{i=1}^n\sigma_iu_i.$$ 
			Moreover, we define
				\begin{equation*}
					\begin{array}{ll}
						\displaystyle M_1:=\max_{\sigma\in S^{n-1}} \|\n u_\sigma\|_2^2,\quad
						&\displaystyle M_2:=\max_{\sigma\in S^{n-1}} \big(\|u_\sigma\|_2^4 + |V_{0,2}(u_\s)| \big),\vspace{9pt}\\
						\displaystyle m_3:=\min_{\sigma\in S^{n-1}} \|u_\s\|_q^{2q},\quad
						&\displaystyle M_4:=\displaystyle\max_{\sigma\in S^{n-1}} |V_{0,q}(u_\s)|
					\end{array}
					\end{equation*}
			and,
			for $t>0$, $u_{\sigma t}:=tu_\sigma(\cdot/t)$.\\
			If $t>1$, we have
		\begin{align*}
			I(u_{\sigma t})
			&=
			t^2\|\n u_\sigma\|_2^2
			+  \frac{t^8\log t}{2} \|u_\sigma\|_2^4 +\frac{t^8}{2} V_{0,2}(u_\sigma)
			- \frac{t^{2q+4}\log t}{q} \|u_\sigma\|_q^{2q}- \frac{t^{2q+4}}{q} V_{0,q}(u_\sigma) \\
			&\le M_1t^2 + \frac {M_2}2t^8(\log t +1) 
			+ \frac{t^{2q+4}}{q}\big(M_4-m_3 \log t\big).
			\end{align*}
			Then there exists $\bar t>1$ such that $\sup_{\s\in S^{n-1}}I(u_{\s\bar t})<0$.\\ 
			We just define 
			\begin{equation*}
			\g_{0n}:\s\in S^{n-1}\mapsto u_{\s\bar t}\in \X_r
			\end{equation*}
			and conclude.
		\end{proof}

		Now, in a similar way as in \cite{A2},  we give the following definition
		\begin{definition}
			Suppose that $(E,\|\cdot\|_E)$ and
			$(F,\|\cdot\|_F)$ are two Banach spaces such that $(E,\|\cdot\|_E)\hookrightarrow (F,\|\cdot\|_F)$.
			A functional $I\in C^1(E,\R)$ satisfies a weak Palais-Smale (resp. Cerami) condition with respect to $E$ and $F$ in the interval $H\subset \R$
			if for any sequence $\{x_n\}$ in $E$  such that
			\begin{itemize}
				\item[$1.$] $\{I(x_n)\}$ is bounded in $H$,
				\item[$2.$] $I'(x_n)\to 0$ in $E'$ (resp. $\|I'(x_n)\|(1+\|x_n\|_E)\to 0$),
				\item[$3.$] $\{x_n\}$ is bounded in $F$,
			\end{itemize}
			there exists a converging subsequence (in the topology of
			$E$).
		\end{definition}

	In our case we have

	\begin{lemma}\label{WC}
The functional $I$ satisfies a weak Cerami condition with respect to $\X_r$ and $\Hr$ in $(-\infty,-\delta)$ and in $(\delta,+\infty)$ for all $\d>0$.
	\end{lemma}
	\begin{proof}
	Assume that $\{u_n\}$ is a sequence in $\X_r$ such that
			\begin{enumerate}[label=\arabic*.]
				\item $\{I(u_n)\}$ is bounded either in $(-\infty,-\d)$ or in $(\d,+\infty)$,
				\item $\|I'(u_n)\|_{\X_r'}(1+\|u_n\|_{\X_r})\to 0$,
				\item $\{u_n\}$ is bounded in $\Hr$.
			\end{enumerate}
Therefore, up to a subsequence, there exists $u\in \Hr$ such that $u_n\rightharpoonup u$ weakly in $\Hr$. 
We claim that $u\neq 0$.
\\
			Suppose by contradiction that, up to a subsequence, $u_n\rightharpoonup 0$ in $\Hr$. By compact embedding  $u_n \to 0$ in $L^\tau(\RD)$, for any $\tau>2$, as $n \to +\infty$. So, by \eqref{control 2},  $V_{2,2}(u_n) \to 0$ and $V_{2,q}(u_n) \to 0$, while, by Proposition \ref{cov1} also $V_{1,q}(u_n)\to 0$, as $n \to +\infty$.
			\\
			Therefore we get
			\[
			\|\n u_n\|^2_2+V_{1,2}(u_n) =\frac 12 I'(u_n)[u_n]+V_{2,2}(u_n)+ V_{1,q}(u_n) -V_{2,q}(u_n) 
			=o_n(1). 
			\]
			This implies that $I(u_n)\to 0$ and we get a contradiction with the fact that $|I(u_n)|\ge\d >0$, as $n \to +\infty$.
		\\
		Hence we can find a nontrivial $u\in \Hr$ such that $u_n \weakto u$ weakly in $\H$ and $u_n \to u$  pointwise almost everywhere. 
		By Proposition \ref{levikx} and Proposition \ref{cov1}, we infer that $\{V_{2,2}(u_n)\}, \{V_{1,q}(u_n)\}$, and $\{V_{2,q}(u_n)\}$ are bounded, and so, since $I'(u_n)[u_n]=o_n(1)$, we deduce that also  $\{V_{1,2}(u_n)\}$ is bounded.
		\\
		Applying Lemma \ref{lemma2.1} with $\{v_n\}=\{u_n\}$, there exists $C>0$ such that $\|u_n\|_{*,2}\le C$ and so $\{u_n\}$ is bounded in $\X_r$. By Proposition \ref{embX}, we deduce that $u_n \to u$  in $L^\tau(\RD)$, for all $\tau\ge 2$.\\
		Therefore, for $k\ge 2$,  by Hardy-Littlewood-Sobolev inequality,
		\begin{align*}
		|V_{2,k}'(u_n)[u_n -u]|&\le 
		C \ird \ird \frac 1{|x-y|}|u_n(x)|^k|u_n(y)|^{k-1}|u_n(y)-u(y)|\,dx\,dy
		\\
		&\le C\|u_n\|_{\frac{4k}3}^{2k-1}\|u_n-u\|_{\frac{4k}3}=o_n(1).
		\end{align*}
		Moreover, using also Lemma \ref{strauss},
		\begin{align*}
		&|V_{1,q}'(u_n)[u_n -u]|
		\le C\ird \ird \log(2+|x-y|)|u_n(x)|^q |u_n(y)|^{q-1}\big|u_n(y)-u(y)\big|\,dx\,dy
		\\
		&\le C\ird \ird \log(2+|x|)|u_n(x)|^q |u_n(y)|^{q-1}\big|u_n(y)-u(y)\big|\,dx\,dy
		\\
		&\quad+ C\ird \ird \log(2+|y|)|u_n(x)|^q |u_n(y)|^{q-1}\big|u_n(y)-u(y)\big|\,dx\,dy
		\\
		&= C\|u_n\|_{*,q}^q   \ird |u_n(y)|^{q-1}\big|u_n(y)-u(y)\big|\,dy
		\\
		&\quad+ C\|u_n\|_q^{q} \ird \log(2+|y|) |u_n(y)|^{q-1}\big|u_n(y)-u(y)\big|\,dy
		\\
		&\le C\|u_n\|^q\|u_n\|_q^{q-1}\|u_n-u\|_q
		\\
		&\quad
		+ C\|u_n\|_q^{q}\left( \ird \big(\log(2+|y|)\big)^{\frac q{q-1}} |u_n(y)|^{q}\,dy\right)^{\frac{q-1}q}\|u_n-u\|_q
		\\
		&\le o_n(1)
		+C\left(\! \int_{B_1}  \!|u_n(y)|^{q}\,dy+ \|u_n \|^{q-2}\!\int_{B_1^c} \!\frac{\big(\log(2+|y|)\big)^{\frac q{q-1}}}{|y|^{\frac{q-2}2}} |u_n(y)|^{2}\,dy\right)^{\frac{q-1}q}\!\!\!\!\|u_n-u\|_q.
		\\
		&=o_n(1).
		\end{align*}
		On the other hand, by Lemma \ref{lemma2.6}, we have 
		\begin{align*}
		V_{1,2}'(u_n)[u_n -u]& 
		= 4\ird \ird \log(2+|x-y|)u_n^2(x) u_n(y)\big(u_n(y)-u(y)\big)\,dx\,dy
		\\
		&= 4\ird \ird \log(2+|x-y|)u_n^2(x) \big(u_n(y)-u(y)\big)^2\,dx\,dy
		\\
		&\quad+ 4\ird \ird \log(2+|x-y|)u_n^2(x) u(y)\big(u_n(y)-u(y)\big)\,dx\,dy
		\\
		&\ge  4\ird \ird \log(2+|x-y|)u_n^2(x) u(y)\big(u_n(y)-u(y)\big)\,dx\,dy=o_n(1).
		\end{align*}
		Observe, in addition, that, since $\{u_n\}$ is bounded in $\X$,
		\begin{equation*}
		|I'(u_n)[u_n -u]|
		\le \|I'(u_n)\|_{\X'}\|u_n-u\|_{\X}=o_n(1).
		\end{equation*}
		Therefore, putting all these informations together, we get
		\begin{align*}
		o_n(1)=I'(u_n)[u_n -u]
		\ge \|\n u_n\|_2^2-\|\n u\|_2^2+o_n(1)\ge o_n(1),
		\end{align*}
		so $u_n \to u$ strongly in $\H$ and since 
		\begin{align*}
		o_n(1)=I'(u_n)[u_n -u]
		= 2\ird \ird \log(2+|x-y|)u_n^2(x) \big(u_n(y)-u(y)\big)^2\,dx\,dy+o_n(1),
		\end{align*}
		using once again Lemma \ref{lemma2.1}, we infer that
		\[
		\|u_n-u\|_{*,2}=o_n(1).
		\]
		This implies that $u_n \to u$ strongly in $\X_r$.		
	\end{proof}

		\begin{proposition}\label{pr:C}
			The functional $G$ satisfies the Cerami condition  in $(\d,+\infty)$ for every $\d>0$.
		\end{proposition}
		
		\begin{proof}
			First observe that, for all $\d>0$, the functional $G$ satisfies a weak Cerami condition with respect to $\X_r$ and $\Hr$ in the interval $(\d,+\infty)$ since the proof of the previous result can be exactly repeated for $G$ in the place of $I$,  ignoring estimates for $V_{2,q}$ and its differential.\\
			Thanks to this, we only have to show that if  $\{u_n\}$ is a Cerami sequence of $G$ such that $\inf_nG(u_n)>0$, then it is bounded in $\Hr$.
			So, take such a sequence and observe that, since
			\begin{equation*}\label{PS}
			G(u_n)+o_n(1)= G(u_n)-\frac 1{4}  G'(u_n)[u_n]=\frac{1}{2} \|\n u_n\|_2^2+\frac{q-2}{2q}V_{1,q}(u_n),
			\end{equation*}
			we have  that $\{\|\n u_n\|_2\}$ and $\{V_{1,q}(u_n)\}$ are bounded. As a consequence, observing that
			\begin{equation*}
			G(u_n)+o_n(1)= G(u_n)-\frac 1{2q}  G'(u_n)[u_n]=\frac{q-1}{q} \|\n u_n\|_2^2+\frac{q-2}{2q}V_{0,2}(u_n),
			\end{equation*}
			we deduce that also $\{V_{0,2}(u_n)\}$ is bounded.
			\\
			So there exists $M>0$ such that
			\begin{equation*}
			\label{secondestim}
			\begin{split}
			V_{1,2}(u_n)&\le M+ V_{2,2}(u_n)\\
			&\le
			M+2 \ird \ird \frac{1}{|x-y|}|u_n(x)|^2|u_n(y)|^2\, dx \, dy\\
			&\le M+
			C \|u_n\|_q^2 \|u_n\|_\frac{4q}{3q-4}^2\\
			&\le
			M+C
			\Big( \|u_n\|_q^{2q} 
			+ \|u_n\|_\frac{4q}{3q-4}^\frac{2q}{q-1}\Big)\\
			&\le
			M+C \|u_n\|_q^{2q} 
			+C
			\|\nabla u_n\|_2^\frac{4-q}{q-1}
			\|u_n\|_2 ^\frac{3q-4}{q-1}\\
			&\le M+
			C  \|u_n\|_q^{2q}
			+C 
			\left(
			\|\nabla u_n\|_2^2
			+
			\|u_n\|_2^\frac{2(3q-4)}{3(q-2)}
			\right)\\
			&\le M+
			C V_{1,q}(u_n)
			+C \|\nabla u_n\|_2^2
			+C\|u_n\|_2^\frac{2(3q-4)}{3(q-2)}
			\end{split}
			\end{equation*}
			and, by boundedness of $\{\|\n u_n\|_2\}$, $\{V_{1,q}(u_n)\}$ and \eqref{eq:controlL2}, we deduce that $\{\|u_n\|_2\}$ is bounded since  $q\in(8/3,4)$ .
		\end{proof}
		
		Now,  for every $n\ge 1$, we define 
		\begin{align}
		b_n&=\inf_{\g\in\G_n}\max_{\sigma\in D_n}I(\g(\sigma))\label{bn}\\
		c_n&=\inf_{\g\in\G_n}\max_{\sigma\in D_n}G(\g(\sigma))\nonumber
		\end{align}
		where $D_n=\{\sigma\in\R^n\mid |\sigma|\le1\}$ and 
		$$\G_n=\{\g: D_n\to \X_r \mid \g \hbox{ is odd and continuous, and } \g{|_{\partial D_n}}=\g_{0,n} \},$$
where $\g_{0,n}$ is defined in Proposition \ref{geo}.

	Moreover the following holds.
		\begin{proposition}\label{useful}
			We have
			\begin{enumerate}[label=(\roman*),ref=\roman*]
				\item \label{primorigo}  $b_n\ge c_n\ge \mu$,
				\item \label{secondorigo} $\lim_n c_n=+\infty$.
			\end{enumerate}
		\end{proposition} 
		
\begin{proof}
	Item (\ref{primorigo}) is a consequence of Proposition \ref{geo}. For (\ref{secondorigo}) we can argue as in \cite{HIT}, observing that the well known minimax methods described in \cite{R2} hold also in our case. Indeed, due to Proposition \ref{pr:C}, one can show that the subsets $K_{c}$, $K_{\overline{c}}$, and $\mathcal{K}$ in \cite[Chapter 9]{R2} are compact.
\end{proof}

		In the sequel we follow  the strategy of  \cite{HIT} (see also\cite{jj}).
		
		Define $\Tilde{\X_r}:=\R \times \X_r$ equipped with the norm $\|(s,v)\|_{\Tilde{\X_r}}:=(s^2+\|v\|^2)^{1/2}$. For any $\a\in \R$,
		define
			\[
	\rho_\alpha:\Tilde{\X_r}\to \X_r,
	\qquad
	\rho_\alpha(s,v):=e^{\alpha s}v(e^s\cdot),
		\]
and  $\varphi_\alpha:\Tilde{\X_r}\to \R$ 
 such that
	\begin{align*}
	\varphi_\alpha (s,v)
	&:=
	e^{2\alpha s} \|\n v\|^2_2 
	+ \frac {e^{4(\alpha-1)s}}{2}\ird \ird \log (|x-y|)|v(x)|^2|v(y)|^2\, dx \, dy
	-\frac {se^{4(\alpha-1)s}}{2} \|v\|_2^{4}
	\\
	&\quad 
	-\frac{e^{2(\alpha q-2)s}}{q}\ird \ird \log (|x-y|)|v(x)|^q|v(y)|^q\, dx \, dy
	+\frac {se^{2(\alpha q-2)s}}{q}  \|v\|_q^{2q}.
	\end{align*}
Observe that $\varphi_\a\in C^1$ and  for all $(s,v)\in\Tilde\X_r$
$$ I(\rho_\a(s,v))=\varphi_\a(s,v).$$
	Finally we define $J_\a:\X_r\to \R$ as follows
	\begin{align*}
	J_\alpha(u)
	&:=
	2\alpha  \| \nabla u \|_2^2 
	+ 2 (\alpha -1) \ird \ird \log (|x-y|)|u(x)|^2|u(y)|^2\, dx \, dy
	-\frac {1}{2} \| u \|_2^4\\
	&\qquad
	-\frac{2(\alpha q-2)}{q} \ird \ird \log (|x-y|)|u(x)|^q |u(y)|^q \, dx \, dy 
	+\frac {1}{q} \|u\|_q^{2q}.
	\end{align*}
	Observe that
	\begin{equation}
	\label{partialsphi}
	\begin{split}
	\partial_s \varphi_\alpha (s,v)
	&=
	2\alpha e^{2\alpha s} \|\n v\|^2_2 \\
	&\quad
	+ 2(\alpha-1)e^{4(\alpha-1)s} \ird \ird \log (|x-y|)|v(x)|^2|v(y)|^2\, dx \, dy\\
	&\quad
	-\frac {(1+4(\alpha-1)s)e^{4(\alpha-1)s}}{2} \|v\|_2^{4}
	\\
	&\quad 
	-\frac{2(\alpha q-2)}{q}e^{2(\alpha q-2)s}\ird \ird \log (|x-y|)|v(x)|^q|v(y)|^q\, dx \, dy\\
	&\quad
	+\frac {1+2(\alpha q-2)s}{q} e^{2(\alpha q-2)s}  \|v\|_q^{2q}\\
	&=
	2\alpha  \ird |\n \rho_\alpha (s,v) |^2 dx \\
	&\quad
	+ 2 (\alpha -1) \ird \ird \log (|x-y|)|\rho_\alpha(s,v(x))|^2|\rho_\alpha(s,v(y))|^2\, dx \, dy \\
	&\quad
	-\frac {1}{2} \left(\ird |\rho_\alpha(s,v)|^2 dx\right)^{2}\\
	&\quad
	-\frac{2(\alpha q-2)}{q} \ird \ird \log (|x-y|)|\rho_\alpha(s,v(x))|^q|\rho_\alpha(s,v(y))|^q\, dx \, dy \\
	&\quad 
	+\frac {1}{q} \left(\ird |\rho_\alpha(s,v)|^q dx\right)^{2}
	= J_\alpha (\rho_\alpha(s,v))
	\end{split}
	\end{equation}
	Moreover, by the linearity of $\rho_\alpha$ with respect to the function in $\X_r$,
	\begin{equation}
	\label{partialvphi}
	\partial_v  \varphi_\alpha (s, v) [w]=I'(\rho_\alpha(s, v))[\rho_\alpha(s, w)].
	\end{equation}
	Now, for any $n\ge 1,$ let us consider
		\[
		\tilde b_{\a,n}
		=
		\inf_{\tilde{\gamma} \in \tilde{\Gamma}_{n}} \max_{\s\in D_n} \varphi_\alpha(\tilde{\gamma}(\s)),
		\]
		where

		\[
		\tilde{\Gamma}_{n}
		:=
		\left\{
		\tilde{\gamma} \in C(D_n,\:\tilde X_{r}) :
		\begin{array}{l}
			\tilde{\gamma}=(s,\eta), s:D_n\to\R {\hbox{ is even}},\\
			 \eta:D_n\to \X_r \hbox{ is odd}, \tilde{\gamma}_{|\partial D_n}=(0,\g_{0,n})
		\end{array}
		\right\}.
		\]
		Observe that $\tilde\g=(0,\g)\in\tilde\Gamma_n$, for any $\g\in\Gamma_n$.
		On the other hand, for any $\a\in\R$ and for any $\tilde \g =(s,\eta)\in\tilde \Gamma_n$, $\g:=e^{\a s(\cdot)}\eta(\cdot)(e^{s(\cdot)}x)\in\Gamma_n $.

		As in \cite[Lemma 4.1]{HIT} we deduce the following
\begin{lemma}\label{bb}
For every $\a\in\R$ and $n\ge1$ we have $b_n=\tilde b_{\a,n}$.
\end{lemma}
\begin{proof}
If $\gamma\in\Gamma_n$ and $\s\in D_n$, 
		\[
		I(\gamma(\s))= I(\rho_\alpha(0,\gamma(\s)))=\varphi_\alpha(0,\gamma(\s)),
		\]
		then
		\begin{equation}
		\label{max0}
		\max_{\s\in D_n
		} I(\gamma(\s))= \max_{\s \in D_n}\varphi_\alpha(0,\gamma(\s))
		\end{equation}
		and so
		\[
		b_n = \inf_{\gamma \in \Gamma_n} \max_{\s\in D_n}I(\gamma(\s))= \inf_{\gamma \in \Gamma_n} \max_{\s\in D_n}\varphi_\alpha(0,\gamma(\s)) \ge \inf_{\tilde\gamma \in \tilde\Gamma_n} \max_{\s\in D_n}\varphi_\alpha(\tilde\gamma(\s))= \tilde b_{\a,n}.
		\]
		On the other hand, if
		$\tilde{\gamma}\in\tilde{\Gamma}_n$,  for every $\s\in D_n$, $\varphi_\alpha(\tilde{\gamma}(\s))=I(\rho_\alpha\circ\tilde{\gamma}(\s))$ so that
		\[
		\max_{\s \in D_n}\varphi_\alpha(\tilde{\gamma}(\s))
		= \max_{\s\in D_n} I(\rho_\alpha\circ\tilde{\gamma}(\s))
		\]
		and then
		\[
		\tilde b_{\a,n} =\inf_{\tilde\gamma \in \tilde{\Gamma}_n} \max_{\s\in D_n} \varphi_\a(\tilde\g(\s))= \inf_{\tilde\gamma \in \tilde{\Gamma}_n} \max_{\s\in D_n} I(\rho_\alpha\circ\tilde{\gamma}(\s))  \ge \inf_{\gamma \in \Gamma_n} \max_{\s\in D_n}I(\gamma(\s))=b_n. 
		\]
\end{proof}
		
		Now we are going to prove that every $b_n$ is a critical level. \\
		We first introduce  the following abstract result

	\begin{proposition}[Proposition 2.8 of \cite{LW}]\label{LWProp28}
		Let $X$ be a Banach space and $M$ a metric space. Let $M_0$ be a closed subspace of $M$ and $\Gamma_0 \subset C\left(M_0, X\right)$. Define
		$$
		\Gamma:=\left\{\gamma \in C(M, X):\left.\gamma\right|_{M_0} \in \Gamma_0\right\}
		$$
		If $\varphi \in C^1(X, \mathbb{R})$ satisfies
		$$
		\infty>c:=\inf _{\gamma \in \Gamma} \sup _{u \in M} \varphi(\gamma(u))>a:=\sup _{\gamma_0 \in \Gamma_0} \sup _{u \in M_0} \varphi\left(\gamma_0(u)\right),
		$$
		then, for every $\varepsilon \in\left(0, \frac{c-a}{2}\right), \delta>0$, and $\gamma \in \Gamma$ such that
		$$
		\sup _M \varphi \circ \gamma \le c+\varepsilon
		$$
		there exists $u \in X$ such that
		\begin{enumerate}[label=(\alph*),ref=\alph*]
			\item $c-2 \varepsilon \le \varphi(u) \le c+2 \varepsilon$;
			\item $\operatorname{dist}(u, \gamma(M)) \le 2 \delta$;
			\item $\left(1+\|u\|_X\right)\left\|\varphi^{\prime}(u)\right\|_{X^{\prime}}<\frac{8 \varepsilon}{\delta}$.
		\end{enumerate}
	\end{proposition}
	
	\begin{lemma}\label{LemexCPS}
For every $n\ge 1$ and $\a\in\R$, there exists a sequence $\{u_{j}\}\subset\X_r$ such that, up to a subsequence,	
\begin{itemize}
\item[$1.$] 	
$I(u_{j})\to b_n$,
\item[$2.$] $\|I'(u_{j})\|_{\X'}\left(1+\|u_{j}\|_{\X_r}\right)\to 0$,
\item[$3.$] 
$J_\alpha(u_{j}) \to 0$,
\end{itemize}
 as  $j\to+\infty$.

	\end{lemma}

	\begin{proof}
		By definition of $b_n$, for any $j\ge 1$ there exists  $\gamma_j\in\Gamma_n$ such that
		\[
		\max_{\s \in D_{n} } I(\gamma_j(\s))\le b_n + \frac{1}{j^2},
		\]
		so that, taking $\tilde{\gamma}_j=(0,\gamma_j)\in\tilde{\Gamma}_n$, by \eqref{max0}, we have
		\[
		\max_{\s \in D_{ n} }  \varphi_\alpha(\tilde{\gamma}_j(\s))\le b_n + \frac{1}{j^2}.
		\]
 Then,  applying Proposition \ref{LWProp28} for $X=\tilde\X_r$, $M=D_{n}$, $M_0= \partial D_{ n}$,
			$\Gamma_0=\{\tilde\g_{0,n}\}$, $\G=\tilde\G_n$, and 
		$\varphi=\varphi_\alpha$, since by Proposition \ref{geo}, Proposition \ref{useful} and Lemma \ref{bb}
			$$\inf_{\tilde\gamma \in \tilde{\Gamma}_n} \max_{\s\in D_{n}} \varphi_\a(\tilde\g(\s))=\tilde {b}_{\a, n}=b_n\ge\mu>0\ge
 \sup _{\sigma \in \partial D_{ n}} \varphi_\a\left(\tilde\gamma_{0,n}(\s)\right),$$
		there exists  $(s_{j},v_j)\in \tilde \X_r$ such that, as $j \to +\infty$,
		\begin{enumerate}[label=($\varphi_\alph*$),ref=$\varphi_\alph*$]
			\item \label{phia} $\varphi_\alpha (s_{j}, v_{j})\to b_n$;
			\item \label{phib} $\operatorname{dist}\big((s_{j},v_{j}), \tilde{\gamma}_{j}(D_{ n})\big)  \to 0$;
			\item \label{phic} $(1+\|(s_{j}, v_{j})\|_{\tilde{\X}_r})\|\varphi_\alpha'(s_{j}, v_{j})\|_{\tilde{\X}'} \to 0$.
		\end{enumerate}
		By (\ref{phib}) we get that 
		\begin{equation}\label{sn0}
		s_{j}\to 0.
		\end{equation}
		Moreover, by \eqref{partialsphi} and \eqref{partialvphi}, for every $(s,v)\in\tilde{\X_r}$,
		\begin{equation}
		\label{phi'}
		\varphi_\alpha'(s_{j}, v_{j}) [(s,v)]
		=I'(\rho_\alpha (s_{j}, v_{j})) [\rho (s_{j}, v)]
		+ J_\alpha(\rho_\alpha (s_{j}, v_{j}))s,
		\end{equation}
		and, taking  $s=1$ and $v=0$, by (\ref{phic}), we get
		\[
		J_\alpha(\rho_\alpha (s_{j}, v_{j})) \to 0 \text{ as }\ j\to + \infty.
		\]
		Hence, taking $u_{j}:=\rho_\alpha (s_{j}, v_{j})$, we have that, as $j \to + \infty$,
		\[
		I(u_{j})\to b_n
		\quad
		\text{and}
		\quad
		J_\alpha(u_{j})\to 0.
		\]
		Thus, since we can see any function $v\in\X_r$ as $\rho_\alpha(s_j,\tilde{v}_j)$, where $\tilde{v}_j:=e^{-\alpha s_j} v(e^{-s_j}\cdot)$, by \eqref{phi'} we have
		\[
		I'(u_j)[v]
		= I'(\rho_\alpha(s_j,v_j))[\rho_\alpha(s_j,\tilde{v}_j)]
		=\varphi_\alpha'(s_j,v_j)[(0,\tilde{v}_j)],
		\]
		and so, by \eqref{phic}, \eqref{sn0} and \eqref{phi'},
		\begin{align*}
		(1+\|u_j\|_{\X})
		|I'(u_j)[v]|
		&=\big(1+\|(s_j,  v_j)\|_{\tilde{\X}}+o_j(1)\big)|\varphi_\alpha'(s_j,v_j)[(0,\tilde{v}_j)]|\\
		&=o_j(1)\|\tilde v_j\|_{\X}=o_j(1)\| v\|_{\X}
		\end{align*}
	and we conclude.
	\end{proof}

\begin{lemma}\label{lebound}
		If
		\begin{equation}\label{pluti}
			\a <-\frac{1}{2 (q-2)\log 2},
		\end{equation}
		then, for every $n\ge 1$, we have that
		any sequence $\{u_{j}\}$ in $\X_r$ that satisfies
		\begin{itemize}
			\item[$1.$] $I(u_{j})\to b_n<+\infty$,
			\item[$2.$]	$	\|I'(u_{j})\|_{\X'}\left(1+\|u_{j}\|_{\X}\right)\to 0,$
			\item[$3.$] $J_\alpha(u_{j}) \to 0,$
		\end{itemize}
		as $j\to +\infty$,	is bounded in $\Hr$.
\end{lemma}

	\begin{proof}
	Since
		\begin{align*}
			&I(u_j)  - \frac{1}{2(\alpha q-2)} J_\alpha (u_j)\\
			&=
			\frac{\a q-2-\a}{\a q -2} \|\n u_j\|^2_2 
			+ \frac{\a (q-2)}{2(\a q -2)} \ird \ird \log (|x-y|)|u_j(x)|^2|u_j(y)|^2\, dx \, dy
			\\
			&\quad
			+\frac{1}{4(\a q-2)}   \|u_j\|_2^{4}
			-\frac{1}{2q(\alpha q-2)}  \|u_j\|_q^{2q}
			= b_n+o_j(1),
		\end{align*}
		using  \eqref{Vkrel},
		\begin{equation}
			\label{starting}
			\begin{split}
				&\frac{\a q-2-\a}{\a q -2} \|\n u_j\|^2_2 
				+ \frac{\a (q-2)}{2(\a q -2)} \ird \ird \log (2+|x-y|)|u_j(x)|^2|u_j(y)|^2\, dx \, dy
				\\
				&\quad
				+\frac{1}{4(\a q-2)}   \|u_j\|_2^{4}
				-\frac{1}{2q(\alpha q-2)}  \|u_j\|_q^{2q}\\
				&=
				b_n
				+ \frac{\a (q-2)}{2(\a q -2)} \ird \ird \log \left(1+\frac{2}{|x-y|}\right) |u_j(x)|^2|u_j(y)|^2\, dx \, dy
				+o_j(1).
			\end{split}
		\end{equation}
		Observe that by
		Hardy-Littlewood-Sobolev, Young, and Gagliardo-Nirenberg inequalities we have that for any $\eps>0$,
		\begin{equation}
			\label{secondestim}
			\begin{split}
				&	\ird \ird \log \left(1+\frac{2}{|x-y|}\right)|u_j(x)|^2|u_j(y)|^2\, dx \, dy\\
				&\le
				2 \ird \ird \frac{1}{|x-y|}|u_j(x)|^2|u_j(y)|^2\, dx \, dy\\
				&\le
				C \|u_j\|_q^2 \|u_j\|_\frac{4q}{3q-4}^2\\
				&\le
				C
				\Big(\varepsilon^q \|u_j\|_q^{2q} 
				+ \frac{1}{\varepsilon^\frac{q}{q-1}}\|u_j\|_\frac{4q}{3q-4}^\frac{2q}{q-1}\Big)\\
				&\le
				C \varepsilon^q \|u_j\|_q^{2q} 
				+C \frac{1}{\varepsilon^\frac{q}{q-1}}
				\|\nabla u_j\|_2^\frac{4-q}{q-1}
				\|u_j\|_2 ^\frac{3q-4}{q-1}\\
				&\le
				C \varepsilon^q \|u_j\|_q^{2q}
				+C \varepsilon^\frac{(5q-2)(q-2)}{(4-q)(q-1)} \|\nabla u_j\|_2^2
				+C\frac{1}{\varepsilon^{\frac{q}{q-1}+\frac{4(q-1)}{3(q-2)}}}\|u_j\|_2^\frac{2(3q-4)}{3(q-2)}.
			\end{split}
		\end{equation}
	Thus, putting  \eqref{secondestim} in \eqref{starting},  taking into account \eqref{eq:controlL2}, and since
	\[
	\frac{\a (q-2)}{2(\a q -2)}>0,
	\]
we get
\begin{equation}
	\label{intermediate1}
	\begin{split}
		&\left(\frac{\a q-2-\a}{\a q -2} - \frac{\a (q-2)}{2(\a q -2)} C \varepsilon^\frac{(5q-2)(q-2)}{(4-q)(q-1)}\right) \|\nabla u_j\|_2^2
		\\
		&
		\quad+ \frac{1+2 \a (q-2)\log 2}{4(\a q -2)}  \|u_j\|_2^{4}
		- \frac{\a (q-2)}{2(\a q -2)} C\frac{1}{\varepsilon^{\frac{q}{q-1}+\frac{4(q-1)}{3(q-2)}}}\|u_j\|_2^\frac{2(3q-4)}{3(q-2)}
		\\
		&\le
		b_n
		+ \left( \frac{\a (q-2)}{2(\a q -2)} C \varepsilon^q +\frac{1}{2q(\alpha q-2)} \right) \|u_j\|_q^{2q}
		+o_j(1).
	\end{split}
\end{equation}
So, since $\a<0$ and $2<q<4$, we can take $\varepsilon$ sufficiently small such that 
	\[
	\frac{\a q-2-\a}{\a q -2} - \frac{\a (q-2)}{2(\a q -2)} C \varepsilon^\frac{(5q-2)(q-2)}{(4-q)(q-1)}>0
	\]
	and
	\[
	\frac{1}{2q(\alpha q-2)}  + \frac{\a (q-2)}{2(\a q -2)} C \varepsilon^q<0.
	\]
	Moreover, by \eqref{pluti},
	\[
	\frac{1+2 \a (q-2)\log 2}{4(\a q -2)}>0,
	\]
	and, being $q>8/3$,
	\[
	\frac{2(3q-4)}{3(q-2)}<4.
	\]
	Hence, by \eqref{intermediate1} we can conclude that $\{u_j\}$ is bounded in $H^1(\mathbb{R}^2)$.	
\end{proof}

	\begin{proof}[Proof of Theorem \ref{main}]

				By Lemma \ref{LemexCPS} and Lemma \ref{lebound}, for any  $n\ge 1$ there exists $\{u^n_j\}$ a sequence in $\X_r$ such that
				\begin{itemize}
					\item[$1.$] $I(u^n_j)\to b_n>0,$ as $j\to +\infty$
					\item[$2.$] $\|I'(u^n_j)\|_{\X_r'}(1+\|u^n_j\|_{\X_r})\to 0,$ as $j\to+\infty$
					\item[$3.$] $\{u_j^n\}$ is bounded (with respect to $j$) in $\Hr$.
				\end{itemize}
		By Lemma \ref{WC}, as $j\to +\infty$, the sequence $\{u^n_j\}$ strongly converges to some ${u^n}\in\X_r\setminus\{0\}$ which is a critical point of $I$. Since the levels $\{b_n\}$ diverge, as $n \to +\infty$, we deduce the multiplicity of solutions.
\\
Finally, the regularity follows by Proposition \ref{regola}.
	\end{proof}

	\section{The generalized problem \eqref{eqp}}\label{4}
	
	In this Section we study \eqref{eqp} for $1<p<q$ and $2p^2/(p+1)<q<2p$.
	
	At least formally, solutions to \eqref{eqp} are critical points of the following functional (still denoted by $I$, with an abuse of notation)
	\begin{equation*}
	I(u)= \|\nabla u\|_2^2
	+\frac{1}{2}V_{1,p}(u)
	-\frac{1}{2}V_{2,p}(u)
	-\frac{1}{q}V_{1,q}(u)
	+\frac{1}{q}V_{2,q}(u).
	\end{equation*}
	Then, for $p>1$, we define $\ch$ as the completion of $C_0^{\infty}(\RD)$ with respect to the norm $$\|\cdot\|_{\ch}=\sqrt{\|\n \cdot\|_2^2+\|\cdot\|_p^2}.$$

	In the following proposition, we study the embedding's properties of $\ch$  The proof is standard (see, for example, \cite[Proposition 2.5]{AP}).
	
	\begin{proposition}\label{pr:embedding}
		The space $\ch$ is continuously embedded into $L^s(\RD)$, for any $s\in [p,+\infty)$.
	\end{proposition}
	
%
%
	
	\begin{remark}\label{reH1}
		Observe that if $p=2$, then $\calh^{2,2}=\H$; while for $1<p< 2$, we have that $\ch \hookrightarrow \H$.
	\end{remark}
	
	In this section we will consider directly the radial case for the problems connected with the lack of compact embeddings and with geometrical properties of the functional, as already seen in Section \ref{funct}. So we consider $\chr$ the set of radial functions in $\ch$ and we prove the following Strauss type lemma.
	\begin{lemma}\label{straussp}
		Assume $p>2.$ Then for any $\tau\in\left(0,\frac 1p\right)$, there exists $C_\tau>0$ and $R_\tau>0$ such that, for all  $u\in \chr$, we have 
		\begin{equation*}
		|u(x)|\le C_\tau\frac{\|u\|_{\ch}}{|x|^\tau},\,\quad\hbox{  for } |x|\ge R_\tau.
		\end{equation*}
		Assume   $1<p\le 2$.  Then there exists $C>0$ and $R>0$ such that, for all  $u\in \chr$, we have 
		\begin{equation*}
		|u(x)|\le C\frac{\|u\|_{\ch}}{|x|^{\frac{6-p}8}},\,\quad\hbox{  for } |x|\ge R.
		\end{equation*}		
	\end{lemma}
	\begin{proof}
		Assume $p>2$. 
		Let $k\in\left(0,\frac 2p\right)$ and consider $u$ a radial function in $C_0^\infty(\RD)$.
		For any $r\ge 0$, we have that 
		\begin{align}\label{derivative}
		\left|\frac{d}{dr}\left(r^ku^2(r)\right)\right|&\le k r^{k-1}u^2(r)+2r^k|u(r)||u'(r)|\nonumber\\
		&\le k r^{k-1}u^2(r) + r^{2k-1}u^2(r) + r|u'(r)|^2. 
		\end{align}
		Now, fix $r\ge 1$ and integrate $-\frac{d}{ds}\left(s^ku^2(s)\right)$ in the interval $[r,+\infty)$. We have
		\begin{align*}
		r^ku^2(r)&\le k\int_r^{+\infty} s^{k-\frac {p+2}p}s^{\frac 2p}u^2(s)\, ds+\int_r^{+\infty}s^{2k-\frac {p+2}p}s^{\frac 2p}u^2(s)\, ds+ \frac{\|\n u\|_2^2} {2\pi}\\
		&\le \frac k{\sqrt[p]{4\pi^2}} \left(\int_r^{+\infty} s^{\frac{pk-p-2}{p-2}}\,ds\right)^{\frac{p-2}{p}}\|u\|_p^2\\
		&\quad+\frac 1{\sqrt[p]{4\pi^2}}\left(\int_r^{+\infty} s^\frac{2pk-p-2}{p-2}\,ds\right)^{\frac{p-2}{p}}\|u\|_p^2+\frac{\|\n u\|_2^2} {2\pi}\\
		&\le C (r^{\frac{pk-4}{p}}+r^\frac{2pk-4}{p})\|u\|_p^2+\frac{\|\n u\|_2^2} {2\pi}
		\le C \|u\|_{\ch}^2.
		\end{align*}
		\\
		Now assume $1<p\le 2$ . If $p=2$ the inequality is nothing but a consequence of the well known Strauss radial lemma since $\calh^{2,2}_r=\Hr$.  If  $1<p<2$, since by Remark \ref{reH1} we know that  $\chr\hookrightarrow \Hr$,  using again \eqref{derivative}, Strauss radial lemma implies that
		\begin{align*}
		r^\frac{6-p}4u^2(r)&\le \frac{6-p}4\int_r^{+\infty} s^{\frac{2-p}4}|u(s)|^{2-p}|u(s)|^p\, ds\\
		&\quad+\int_r^{+\infty}s^{\frac{4-p}2}|u(s)|^{2-p}|u(s)|^p\, ds+ \frac{\|\n u\|_2^2} {2\pi}\\
		&\le C\|u\|^{2-p}_{H^1}\int_r^{+\infty} (s^{-\frac{p+2}4}+s^{-\frac{p-2}2})\frac{ s|u(s)|^p}{s^{\frac{2-p}{2}}}\,ds+\frac{\|\n u\|_2^2} {2\pi}\\
		&\le  \frac {C\|u\|^{2-p}_{H^1}}{2\pi }(1+r^{\frac{p-6}4})\|u\|_p^p+\frac{\|\n u\|_2^2} {2\pi}
		\le C\|u\|^2_{\ch}. 
		\end{align*}
In both the cases the conclusion follows easily by a density argument. 
	\end{proof}
	%

	As a consequence of the previous lemma,
	$\chr$ is continuously embedded into $L^q_{\log(2+|\cdot|)}(\RD)$. More in general, the following holds
	
	\begin{lemma}\label{embHrp}
		For all $\tau\in (p,+\infty)$,  there exists a constant $C$ such that for any $u\in \chr$,
		\begin{equation*}
		\|u\|_{*,\tau}^\tau
		\le C\left(\|u\|_\tau^\tau+\|u\|^{\tau-p}_{\ch}\|u\|_p^p\right)\le C \|u\|^\tau_{\ch}.
		\end{equation*}
		Then the space $\chr$ is continuously embedded into $L^\tau_{\log(2+|\cdot|)}(\RD)$, for any $\tau\in (p,+\infty)$.
	\end{lemma}
	
	\begin{proof}
		Let $u\in \chr$. Using Lemma \ref{straussp}, the claim  follows observing that 
		\begin{align*}
		\|u\|_{*,\tau}^\tau&
		\le  C  \left(\int_{B_1} |u|^\tau\, dx+\|u\|^{\tau-p}_{\ch}\int_{\RD\setminus B_1} \frac{\log(2+|x|)}{|x|^{\rho(\tau-p)}}|u|^p \, dx\right)\\
		&\le C \big(\|u\|_\tau^\tau+\|u\|^{\tau-p}_{\ch}\|u\|_p^p\big)\le C \|u\|^\tau_{\ch}
		\end{align*}
		where, in the last inequalities, $\rho\in (0,1/p)$, if $p>2$ and $\rho=(6-p)/8$, if $1<p\le 2$.
	\end{proof}
	
	By the previous lemma, we see that all the terms of the functional $I$ are well defined in $\chr$ except $V_{1,p}$. For this reason we introduce the space $\X^{2,p}$ obtained as the completion of $C_0^{\infty}(\RD)$ with respect to the norm $$\|\cdot\|_{\X^{2,p}}:=
	\sqrt{\|\n \cdot\|_2^2+\|\cdot\|_{*,p}^2},$$
	where $\|\cdot\|_{*,p}$ is defined in \eqref{norma*k}.
	Clearly, $\X^{2,p}$ is continuously embedded into $\ch$. In addition, we denote by $\X_r^{2,p}$ the subspace of radial functions in $\X^{2,p}$ and it is compactly embedded into $L^\tau(\RD)$ and continuously embedded into $L^\tau_{\log(2+|\cdot|)}(\RD)$, for $\tau\ge p$.

	The following lemma shows that the functional $I$ is well defined  in $\X_r^{2,p}$.
	\begin{lemma}\label{levikp}
		For every $u\in \X_r^{2,p}$, the following inequalities hold
		\begin{align}
		0\le 	V_{1,p}(u)&\le C \|u\|_{*,p}^p\|u\|_p^p,\label{v1p}
		\\
		0\le 	V_{2,k}(u)&\le C \|u\|_{\frac{4k}3}^{2k}, \quad
		\hbox{ for }k=p,q,\label{v2k}
		\\
		0\le 	V_{1,q}(u) &\le   C\|u\|_q^q  \big(\|u\|_q^q+\|u\|^{q-p}_{\ch}\|u\|_p^p\big).\label{v1q}
		\end{align}
		In particular, there exists a positive constant $C>0$ such that,
		for every $u\in \X_r^{2,p}$, we have
		\begin{equation*}
		0\le 	V_{i,k}(u)\le C\|u\|^{2k}_{\cx}, \hbox{ for }k=p,q \hbox{ and } i=1,2.
		\end{equation*}		
	\end{lemma}
	\begin{proof}
		Arguing as in Proposition \ref{levikx} we get \eqref{v1p} and \eqref{v2k}, while  \eqref{v1q} follows by Lemma \ref{embHrp}.
	\end{proof}

	Moreover, arguing again as in \cite[Lemma 2.2]{CW}, the functional $I$ is of class $C^1$ in $\cxr$, the analogous of Lemmas \ref{lemma2.1} and \ref{lemma2.6} hold, and $I$ possesses nice geometry properties in $\cxr$ (see Proposition \ref{geo}).
In analogy with \eqref{bn} we will introduce minimax levels $b_n$.
	
As  in Section \ref{3},  we define $J_\a:\cxr\to \R$ as follows
	\begin{align*}
	J_\alpha(u)
	&:=
	2\alpha  \| \nabla u \|_2^2 
	+ 2 \frac{(\alpha p -2)}{p} \ird \ird \log (|x-y|)|u(x)|^p|u(y)|^p\, dx \, dy
	-\frac {1}{p} \| u \|_p^{2p}\\
	&\qquad
	-\frac{2(\alpha q-2)}{q} \ird \ird \log (|x-y|)|u(x)|^q |u(y)|^q \, dx \, dy 
	+\frac {1}{q} \|u\|_q^{2q}.
	\end{align*}
	
	Then, as in Lemmas \ref{LemexCPS}  and \ref{lebound}, we have
	\begin{lemma}\label{LemexCPSp}
If $1<p<q$ and $2p^2/(p+1)<q<2p$, there exists $\a<0$ such that, for any $n\ge 1$, there exists a sequence $\{u_{j}\}$ in $\cxr$, bounded in $\ch$, such that, up to a subsequence,	
\begin{itemize}
\item[$1.$]
$		I(u_{j})\to b_n$,
\item[$2.$]
	$\|I'(u_{j})\|_{{(\cx)}'}\left(1+\|u_{j}\|_{\cxr}\right)\to 0$,
\item[$3.$]
$		J_\alpha(u_{j}) \to 0$,		
\end{itemize}
as  $j\to+\infty$.
	\end{lemma}
	
	Hence we can conclude obtaining the following existence result in the general case.

	\begin{theorem}\label{mainp}
		If $1<p<q$ and $2p^2/(p+1)<q<2p$, then there exist   infinitely many solutions to problem \eqref{eqp}.
	\end{theorem}

\

\noindent{\bf Acknoledgements.}
The authors thank the referee for the valuable suggestions and helpful comments which further 
improve the content and presentation of the paper. 
\\
The authors are partially supported by  INdAM - GNAMPA Project 2023 ``Metodi variazionali per
alcune equazioni di tipo Choquard".
P.D. and A. P. are partially financed by PRIN 2017JPCAPN ``Qualitative and quantitative aspects of nonlinear PDEs'' and European Union - Next Generation EU - PRIN 2022 PNRR ``P2022YFAJH Linear and Nonlinear PDE's: New directions and Applications".

\
	
\noindent{\bf Data Availability.} Data sharing is not applicable to this article as no datasets were generated or analysed
	during the current study.

\end{document}